\documentclass[11pt]{article}

\usepackage{fullpage, amsmath, amssymb, amsthm, graphicx}

\title{Computing with voting trees}

\author{Jennifer Iglesias \thanks{Department of Mathematical Sciences,
Carnegie Mellon University, Pittsburgh, PA 15213, e-mail: jiglesia@andrew.cmu.edu.
}
\and
Nathaniel Ince \thanks{Department of Mathematical Sciences,
Carnegie Mellon University, Pittsburgh, PA 15213, e-mail: nince@andrew.cmu.edu.
}
\and
Po-Shen Loh\thanks{Department of Mathematical Sciences,
Carnegie Mellon University, Pittsburgh, PA 15213.  E-mail: {\tt
ploh@cmu.edu}. Research supported by NSF grant DMS-1201380, an NSA
Young Investigators Grant and a USA-Israel BSF Grant.}
}
\date{}

\newtheorem{theorem}{Theorem}[section]
\newtheorem*{theorem*}{Theorem}
\newtheorem{definition}[theorem]{Definition}

\newtheorem{lemma}[theorem]{Lemma}
\newtheorem{corollary}[theorem]{Corollary}

\begin{document}
\maketitle

\begin{abstract}
  The classical paradox of social choice theory asserts that there is no
  fair way to deterministically select a winner in an election among more
  than two candidates; the only definite collective preferences are between
  individual pairs of candidates.  Combinatorially, one may summarize this
  information with a graph-theoretic tournament on $n$ vertices (one per
  candidate), placing an edge from $u$ to $v$ if $u$ would beat $v$ in an
  election between only those two candidates (no ties are permitted).  One
  well-studied procedure for selecting a winner is to specify a complete
  binary tree whose leaves are labeled by the candidates, and evaluate it
  by running pairwise elections between the pairs of leaves, sending the
  winners to successive rounds of pairwise elections which ultimately
  terminate with a single winner. This structure is called a \emph{voting
  tree}.
  
  Much research has investigated which functions on tournaments are
  computable in this way.  Fischer, Procaccia, and Samorodnitsky
  quantitatively studied the computability of the \emph{Copeland rule},
  which returns a vertex of maximum out-degree in the given tournament.
  Perhaps surprisingly, the best previously known voting tree could only
  guarantee a returned out-degree of at least $\log_2 n$, despite the fact
  that every tournament has a vertex of degree at least $\frac{n-1}{2}$.
  In this paper, we present three constructions, the first of which
  substantially improves this guarantee to $\Theta(\sqrt{n})$.  The other
  two demonstrate the richness of the voting tree universe, with a tree
  that resists manipulation, and a tree which implements arithmetic modulo
  three.
\end{abstract}

\section{Introduction}

The study of elections is a complex field. When deciding between two
candidates, one may consider several procedures, but the familiar majority
vote can actually be shown to have nice mathematical properties, as
described in detail by May in \cite{majority}.  Unfortunately, these
properties are difficult to extend to a multi-agent election, and far less
is understood in that setting. One approach is to run independent
2-candidate majority elections between pairs of candidates, and select a
winner based upon those results.  Combinatorially, running all such
elections would produce a tournament (an oriented complete graph, in which
every pair of vertices spans exactly one directed edge) summarizing all
pairwise preferences.

Yet even though such a tournament contains much information about the
candidates, selecting a winner given such a tournament, known in the social
sciences as a \emph{tournament solution}, is a difficult question. Indeed,
McGarvey \cite{tournament00} showed that even with a number of voters only
polynomial in $n$ (the number of agents), every possible tournament is
already achievable in this way.  So, there may be many lists of preferences
among multiple agents that fit the same description. One natural method is
to select a candidate that beats as many other candidates as possible in
the pairwise majority elections. This is known as the \emph{Copeland
solution}, and is equivalent to identifying a vertex of maximum out-degree
in the preference tournament.  Such a vertex may not be unique, so it is
desirable to develop a procedure which selects a vertex with similar
properties, while being completely deterministic.  One natural solution is
to construct a \emph{voting tree}, and this paper analyzes the
combinatorial aspects of such structures, inspired by recent work of
Fischer, Procaccia, and Samorodnitsky \cite{votingTrees}.


We first coordinate our terminology with that typically used in the social
science literature.  Let $\mathcal{T}_n$ denote the set of all tournaments
on the vertex set $[n]$.  An \emph{agenda}\/ is a mapping from
$\mathcal{T}_n$ to $[n]$.  A \emph{match}\/ $M_{i,j}$ is a particular
agenda of the form 
\begin{displaymath}
  M_{i,j}(T)=\left\{ 
  \begin{array}{ll}
    i & \text{if $i=j$ or $\overrightarrow{ij} \in T$;} \\ 
    j & \text{otherwise.}
  \end{array}
  \right.
\end{displaymath}

We will refer to $M_{i,j}$ as the \emph{match between $i$ and $j$}.  Using
this as the basic building block, we may now interpret certain labeled
binary trees as agendas.  Indeed, let $S$ be a (not necessarily complete)
binary tree, in which each node has either 0 or 2 children, where each leaf
is labeled with an integer in $[n]$.  We permit multiple leaves to receive
the same label.  The corresponding agenda $A$ is computed as follows.
Given a tournament $T$ on $[n]$, we recursively label each node of the tree
$S$ by the result of the match between its children's labels.  The ultimate
label at the root is the output of this agenda, and such agendas are called
voting trees.  For instance, $M_{i,j}$ may be represented by a three node
binary tree, as shown below. 
\begin{figure}[h]
\includegraphics[scale=1]{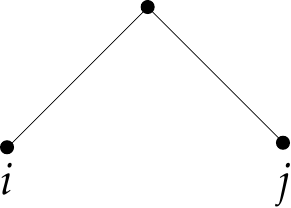} 
\centering
\caption{\footnotesize The $M_{i,j}$ voting tree.}
\end{figure}


Voting trees were first introduced by Farquharson in \cite{vtree00} and
further investigated by various researchers in \cite{vtree01},
\cite{vtree03}, \cite{vtree02}, \cite{vtree04}, and \cite{votingrules}.
Fischer, Procaccia, and Samorodnitsky focused on the problem of
quantitatively approximating the Copeland solution with voting trees. A
voting system which is implementable by a voting tree, implies that the
system is also implementable via backwards induction, as shown in
\cite{backwards}.  Both \cite{socialchoice1} and \cite{socialchoice2}
provide social choice functions which are implementable via backwards
induction.

Let the \emph{performance
guarantee}\/ of a given voting tree $\Theta$ be the minimum out-degree
which it ever produces on any input, i.e.,
\begin{displaymath}
  \min_{T \in \mathcal{T}_n} \{ \text{out-degree of $\Theta(T)$}
  \} \,.
\end{displaymath}

When the number of agents $n$ is a power of two, it is relatively easy to
construct a voting tree with performance guarantee $\log_2 n$: take a
complete binary tree with $n$ leaves, and label the leaves with an
arbitrary permutation of $[n]$.  Then, given any tournament as input, the
ultimate winner must win exactly $\log_2 n$ matches on the way to the root,
and since all leaves are distinctly labeled, the winner must have met a
different opponent at each of those matches.  Therefore, the winner must
always have out-degree at least $\log_2 n$.  On the other hand, for this
voting tree, an adversary can easily construct a suitable tournament to
provide as input such that the winner only has out-degree exactly $\log_2
n$.  In contrast, it is well-known that every $n$-vertex tournament has a
vertex of out-degree at least $(n-1)/2$, so this performance guarantee
appears to be quite weak.

Perhaps surprisingly, prior to our work, this was the best known voting
tree according to this metric, despite non-trivial effort: in
\cite{votingTrees}, the problem of improving this lower bound was even
suspected to currently be ``out of reach.'' Fischer, Procaccia, and
Samorodnitsky also proved the strongest negative result, showing that no
tree can guarantee a winner with out-degree at least $3/4$ of the maximum
out-degree in the input tournament; in particular, quantifiably
strengthening an earlier result of Moulin \cite{tournchoose} which showed
that no voting tree can always identify the Copeland solution. Srivastava
and Trick \cite{votingrules} showed that Moulin's result is also a minimal
example.  However, these lower and upper bounds are extremely far apart,
and even the case $n = 6$ is still not settled.

This article proves a number of positive results which demonstrate the
rather rich computational power of voting trees.  We first exhibit a
construction which substantially improves the logarithmic lower bound of
the trivial construction above.

\begin{theorem}
  There exist voting trees $\Omega_n$ with performance guarantee at least
  $(\sqrt{2} + o(1)) \sqrt{n}$.
  \label{thm:lower}
\end{theorem}

We then show two constructions which we discovered while investigating this
problem.  We feel that they may be of independent interest, to different
communities.  The first is relevant to social choice theory, and achieves a
rather counterintuitive result.  As motivation, observe that any upper
bound for our problem consists of a procedure that takes a voting tree as
input, and produces a suitable tournament as output, such that when tree is
applied to the tournament, the resulting winner has low out-degree.  The
most natural candidate tournaments have the following form.

\begin{definition}
  A \textbf{perfect manipulator tournament} is an $n$-vertex tournament
  with vertex set $\{\alpha\} \cup B \cup C$, where $B$ and $C$ are
  nonempty, the single element $\alpha$ (the ``manipulator'') defeats every
  member of $B$, every member of $B$ defeats every member of $C$, and every
  member of $C$ defeats $\alpha$.
\end{definition}

One might suspect that given any voting tree, if one wants to design a
tournament for which a particular vertex $\alpha \in [n]$ is the winner,
this should be accomplished by having all vertices that $\alpha$ defeats
able to beat all vertices that win against $\alpha$, i.e., by constructing
a suitable perfect manipulator tournament.  The task of improving the upper
bound would then be reduced to specifying and controlling the size of the
set $B$ of vertices defeated by the manipulator $\alpha$.  Surprisingly, this
intuition is completely incorrect.

\begin{theorem}
  \label{thm:manipulator} 
  For every integer $n \geq 4$ which is a power of two, there is a voting
  tree $\Psi$ such that for every $\alpha \in [n]$ and every perfect
  manipulator tournament $T$ with $\alpha$ as the manipulator, $\alpha$
  does not win when $\Psi$ is applied to $T$.
\end{theorem}

Note that the result holds regardless of how large $B$ is, and in
particular, even if $C$, the set of agents who defeat $\alpha$, has size as
small as 1.  The result obviously cannot hold for all $n$, because, for
example, for $n = 3$, one can construct a suitable perfect manipulator
tournament for any tree $\Theta$ by letting $T$ be the unlabeled directed
3-cycle, and feeding it into $\Theta$; whichever vertex wins is labeled as
the manipulator $\alpha$, and $B$ and $C$ are assigned cyclically.  Ad-hoc
constructions exist for several other values of $n$, but these are not
easily generalizable.

\medskip

Our final construction demonstrates that voting trees can be harnessed for
computation.  The first voting tree we discovered that beat the logarithmic
lower bound was built using ``arithmetic gates'' which operated on the most
fundamental case $n=3$.  This is a voting tree analogue of the classical
arithmetic circuits built from AND and OR logical gates.

\begin{theorem}
  There exist voting trees $\Sigma$ and $\Pi$, with leaves labeled by $\{0,
  1, 2, X, Y\}$, with the following properties.  Let $T$ be an arbitrary
  non-transitive tournament on the vertex set $\mathbb{F}_3 = \{0, 1, 2\}$,
  i.e., one of the two cyclic orderings of the vertex set.  Let $x$ and $y$
  be arbitrary elements of $\mathbb{F}_3$.  Then, if in $\Sigma$ and $\Pi$,
  all occurrences of $X$ and $Y$ are replaced by $x$ and $y$, and both
  trees are evaluated on $T$, their corresponding outputs are precisely the
  elements in $\mathbb{F}_3$ corresponding to $x+y$ and $xy$, respectively.
  \label{thm:arithmetic}
\end{theorem}

These trees $\Sigma$ and $\Pi$ can be interpreted as addition and
multiplication gates in the context of voting trees, where the $X$ and $Y$
represent the two inputs to each gate.  These gates can then be chained
together as robust building blocks which independently perform consistently
across input tournaments.  The requirement of non-transitivity is
essential, because if a transitive tournament is input, no non-trivial
voting tree can avoid having the dominant vertex win.


\section{Lower bound}

We now move to describe the constructions.  We start by creating voting
trees $\Omega_n$ with performance guarantees of $(\sqrt{2} + o(1))
\sqrt{n}$, where $n$ is the number of agents.  The following tree serves as
our fundamental building block.

\begin{definition}
  Let $i$ be a label between 1 and $n$ inclusive, and let $S \subset [n]
  \setminus \{i\}$ be a set of other labels.  Construct the complete binary
  tree with exactly $|S|$ pairs of leaves, and for each label $s \in S$,
  place the labels $i$ and $s$ on a distinct pair of leaves, as illustrated
  in Figure \ref{fig:against-S}.  We call the resulting voting tree
  $\Lambda_{i:S}$.
  \label{def:against-S}
\end{definition}

\begin{figure}[h]
\includegraphics[scale=1]{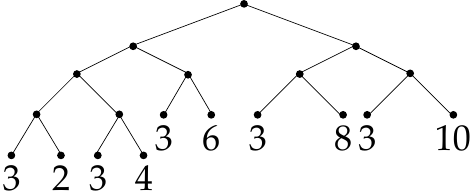} 
\centering
\caption{\footnotesize The $\Lambda_{3:S}$ voting tree for $S = \{2, 4, 6, 8,
10\}$.}
\label{fig:against-S}
\end{figure}

The consequences of placing candidate $i$ in competition with every
candidate in $S$ turn out to be quite convenient.

\begin{lemma}
  When $\Lambda_{i:S}$ is evaluated on a tournament $T$, the result is
  either $i$, if $i$ defeats every element of $S$ according to $T$, or the
  result is some vertex $s \in S$ with the property that $s$ defeats $i$ in
  $T$.
  \label{lem:against-S}
\end{lemma}

\begin{proof}
  The leaves of this tree consist of independent pairwise matches between
  $i$ and every candidate in $S$.  We start by evaluating the pairwise
  matches at the leaves, writing the winners of those $|S|$ matches on
  their parents, and deleting the original leaves.  Each of those winners
  will either be $i$ (if $i$ won the match), or a candidate of $S$ who
  defeats $i$.  In particular, at this point, if $i$ defeats every
  candidate in $S$, then $i$ is the only label that remains, and the winner
  is clearly $i$, as claimed.  Otherwise, all leaves of the tree are now
  labeled with only two types of candidates: $i$ itself, and a nonempty
  sub-collection $R \subset S$ of candidates from $S$ who defeat $i$.

  We now continue to evaluate the tree.  Whenever two elements of $R$ are
  matched against each other, one wins, and one element of $R$ moves on as
  the winner.  However, whenever an element of $R$ is matched against $i$,
  the element of $R$ wins, and moves on as the winner.  Thus, as the
  evaluation continues, there is always at least one element of $R$.  This
  holds all the way up to the root layer, and so we conclude that the
  winner of $\Lambda_{i:S}$ is an element of $R$, i.e., a candidate who
  defeats $i$, as claimed.
\end{proof}

\bigskip

We now construct our voting trees $\Omega_n$ recursively.  The following
lemma steadily improves the performance guarantee as the number of
candidates increases.

\begin{lemma}
  Let $\Omega_n$ be a voting tree for $n$ agents with performance guarantee
  $k$.  Then there is a voting tree $\Omega_{n+k+1}$ for $n+k+1$ agents
  with performance guarantee at least $k+1$.
  \label{lem:lower-induct}
\end{lemma}

\begin{proof}
  To simplify notation, let the $n+k+1$ candidates be indexed $\{0, 1,
  \ldots, n+k\}$.  Let $N = \binom{n+k}{n}$, and consider the tree
  $\Lambda_{0:[N]}$ from Lemma \ref{lem:against-S}, where $[N] = \{1, 2,
  \ldots, N\}$.  For each $i \in [N]$, arbitrarily select a distinct subset
  $S \subset [n+k]$ of size $n$, and remove the leaf labeled $i$ from our
  original $\Lambda_{0:[N]}$.  In its place, insert the tree $\Omega_S$,
  defined as the voting tree $\Omega_n$ for $n$ candidates with performance
  guarantee $k$, but labeled with the $n$ elements of $S$ instead of the
  $n$ elements $\{1, \ldots, n\}$.

  We claim that the resulting tree guarantees that the winner has
  out-degree at least $k+1$.  Indeed, suppose that it has been fed a
  tournament on $\{0, \ldots, n+k\}$ as input.  First, evaluate all
  $\binom{n+k}{n}$ of the $\Omega_S$ subtrees, and write their winning
  candidates at their roots.  The key observation is that at least $k+1$
  distinct candidates emerge from these subtrees.  Indeed, let $W$ be the
  collection of all such intermediate winners.  If $|W| \leq k$, then there
  would be some $n$-set $S' \subset [n+k]$ which was disjoint from $W$.
  The winner of $\Omega_{S'}$ would then be outside $W$, contradiction.

  At this point, we have the structure of $\Lambda_{0:[N]}$, except that
  the set $W$ appears as the non-0-labeled leaves.  By Lemma
  \ref{lem:against-S}, either candidate 0 wins, in which case it defeats
  every candidate in $W$ (producing out-degree at least $k+1$), or some
  candidate $w \in W$ wins, defeating candidate 0 as well.  Since the
  out-degree guarantee for $w$ was already $k$ from its $\Omega_S$, in which
  candidate 0 did not appear, we conclude that $w$ has out-degree at least
  $k+1$, completing the other scenario.

\end{proof}

Our main asymptotic lower bound now follows by induction.

\medskip

\noindent \emph{Proof of Theorem \ref{thm:lower}.}  For $n=2$, we take a
single match between candidates 1 and 2, and it is clear that this
guarantees that the winner has out-degree at least 1.  Repeated application
of the previous lemma then produces voting trees for $\frac{k(k+1)}{2} + 1$
candidates which guarantee out-degree at least $k$.  \hfill $\Box$

\section{Perfect manipulator tournaments}

We now proceed to our second result.  For this construction, let us fix
$n=2^d$ as a power of two.  We begin by introducing compact notation that
we will use to describe various voting trees in this section, all of which
will be complete binary trees with full bottom layers.
We represent such a tree by listing its leaves from left to right as an
$m$-tuple $(x_1, x_2,\ldots, x_m)$ of integers in $[n]$, where $m$ is a
power of two.  For example, the $n$-tuple $(1, 2, \ldots, n)$ corresponds
to the voting tree mentioned in the introduction, which achieved the
previous best performance guarantee of $\log_2 n$.

Using the above notation, we now describe our first gadget, which we will
use as a building block in our construction.  We will show that it has the
property that when it is applied to any perfect manipulator tournament
$\{\alpha\} \cup B \cup C$, then the winner is never in $C$.  It is perhaps
intuitive that such a voting tree might exist, because candidates in $C$
only have a single candidate $\alpha$ who they are capable of defeating,
and therefore have a disadvantage.

\begin{definition}
  Define the permutation $\phi$ of $[n]$ as:
  \begin{displaymath}
    \phi(i) = \left\{
      \begin{array}{ll}
        (i+1)/2 & \text{if $i$ is odd, and} \\ 
        n/2 + i/2 & \text{if $i$ is even.}
      \end{array}
    \right.
  \end{displaymath}
  Let $\Phi$ be the voting tree corresponding to the $2n$-tuple $(1,
  \ldots, n, \phi(1), \ldots, \phi(n))$.
  \label{def:shuffle-tree}
\end{definition}

\begin{figure}[h]
\includegraphics[scale=1]{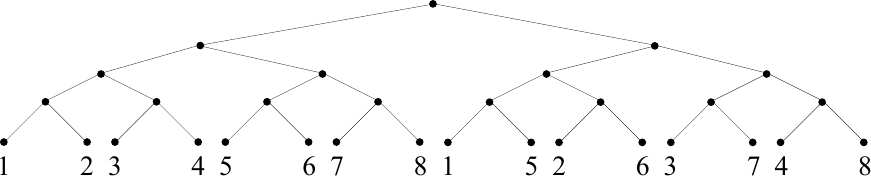} 
\centering
\caption{\footnotesize The voting tree $\Phi$ for $n=8$.}
\label{fig:shuffle-tree}
\end{figure}

This tree is illustrated in Figure \ref{fig:shuffle-tree}. The action of
the permutation is so that the right half of the tree is a perfect shuffle
of the left half of the tree.  The utility of this structure will become
clear in the proof of the following lemma.

\begin{lemma}
  Whenever $\Phi$ is applied to any perfect manipulator tournament, the
  winner is not in $C$. 
\label{lem:phi}
\end{lemma}

\begin{proof}
Suppose for contradiction that $\Phi$ has been applied to such a
tournament, and that the winner is in $C$. The final match of $\Phi$ was
between the winners of its left and right subtrees, which we will refer to
as $\Phi_L$ and $\Phi_R$. Note that by the assumption that a member of $C$
won, the final match must have been between $\alpha$ and an element of $C$,
or between two elements of $C$, so neither subtree's winner can be in $B$. 

Assume that the winner of $\Phi_L$ is in $C$. Note that $\Phi_L$
corresponds to the $n$-tuple $(1,\dots,n)$, as defined at the beginning of
this section. So $\Phi_L$'s leaves contain every element of the tournament
exactly once, and the label $\alpha$ appears in either $\Phi_L$'s left or
right subtree as a leaf. Without loss of generality, suppose that $\alpha$
appears in the left subtree of $\Phi_L$. Any match between an element of
$B$ and something other than $\alpha$ will be won by an element of $B$. So
the right subtree of $\Phi_L$ will be won by an element of $B$ if any
elements of $B$ appear on its leaves, then making it impossible for
$\Phi_L$ to be won by an element of $C$. Thus, the right subtree of
$\Phi_L$'s must have all of its leaves labeled from $C$. 

In $\Phi_R$, each label $1,\dots,\frac{n}{2}$ is initially matched against
a label of $\frac{n}{2}+1,\dots,n$, due to the action of the shuffle
$\phi$. These labels appear on the leaves of the left and right subtrees of
$\Phi_L$, respectively.  So in $\Phi_R$, $\alpha$ is matched with an
element of $C$, and eliminated in the first round. Since at $B$ is
nonempty, the elements of $B$ will continue to dominate every match of
$\Phi_R$, unchallenged by $\alpha$, and an element of $B$ will win
$\Phi_R$. It then will beat the element of $C$ that won $\Phi_L$, making
the winner of $\Phi$ an element of $B$, contradicting our assumptions.

The other possibility is that $\Phi_R$ is won by an element of $C$. The
same argument we used on $\Phi_L$ shows that either the left or the right
subtree of $\Phi_R$ has leaves entirely contained in $C$. Let us suppose
that it is the right subtree of $\Phi_R$, as the other case is similar.
This right subtree contains the nodes $\{\phi(\frac{n}{2}+1),\dots,
\phi(n)\}=\{\frac{n}{4}+1,\dots,\frac{n}{2}\} \cup \{\frac{3n}{4}+1,n\}$.
These are the nodes of the right subtree of the left subtree of $\Phi_L$
and the right subtree of the right subtree of $\Phi_L$. So it's clear that
those trees must have winners in $C$. Therefore $\alpha$ cannot be the
winner of either the left or right subtree of $\Phi_L$.

However, $\alpha$ must be initially matched with an element of $b \in B$ in
$\Phi_R$, because otherwise it will be eliminated in the first round, and
$B$ will win $\Phi_R$ unchallenged, contradicting our assumption that
$\Phi_R$ is won by $C$.  Since $\phi$ alternately maps indices of the left
and right subtrees of $\Phi_L$ into matched pairs, the two subtrees of
$\Phi_L$ have the property that one contains $\alpha$, and the other
contains $b$ and not $\alpha$.  As before, the latter subtree will be won
by $B$, and the previous paragraph's conclusion implies that the other
subtree of $\Phi_L$ is not won by $\alpha$.  Therefore, $\Phi_L$ is won by
$B$, which then defeats the element of $C$ which we assumed as the winner
of $\Phi_R$, contradicting the fact that $\Phi$ was won by $C$.
\end{proof}

Our next gadget enables us to force $\alpha$ to lose by transferring the
status of the ``perfect manipulator'' to another node of $T$.  For each $i
\in [n]$, define $\Lambda_i$ to be the voting tree $\Lambda_{i:S}$ from
Definition \ref{def:against-S}, with $S = [n] \setminus \{i\}$.  To unify
notation, let $A = \{\alpha\}$ in the perfect manipulator tournament. The
same argument as used in Lemma \ref{lem:against-S} shows that the winners
of this tree ``rotate'' the sets $A$, $B$, and $C$.

\begin{lemma}
  When $\Lambda_i$ is evaluated on a perfect manipulator tournament $T$,
  the winner is in $C$ if $i \in A$, in $A$ if $i \in B$, and in $B$ if $i
  \in C$.
  \label{lem:one-against-all}
\end{lemma}

\begin{proof}
  Without loss of generality, suppose that $i \in A$.  After all of the
  initial leaf matches are evaluated, all intermediate winners are either
  from $A$ or $C$, because all candidates from $C$ were immediately
  eliminated by $i$.  Importantly, since all three classes $A$, $B$, and
  $C$ are nonempty in perfect manipulator tournaments, at least one element
  of $C$ survives.  Now all remaining candidates are from either $A$ or
  $C$, and so just as in the proof of Lemma \ref{lem:against-S}, $C$ is
  unchallenged, and an element of $C$ emerges as the winner.
\end{proof}

Next, we compose $\Lambda_i$ with itself to perform a ``double
rotation'' on perfect manipulator tournaments.

\begin{definition}
  For any $i \in [n]$, start with the
  voting tree $\Lambda_i$.  At each leaf, note the label (suppose it
  was $j$), erase the label, and then insert a copy of $\Lambda_j$ with
  its root at that leaf.  Call the result $\Lambda_i^2$.
\end{definition}

\begin{corollary}
  When $\Lambda^2_i$ is evaluated on a perfect manipulator tournament $T$,
  the winner $\Lambda^2_i(T)$ is in $B$ if $i \in A$, in $C$ if $i \in B$,
  and in $A$ if $i \in C$.
  \label{cor:one-against-all-2}
\end{corollary}

\begin{proof}
  Without loss of generality, suppose that $i \in A$.  By Lemma
  \ref{lem:one-against-all}, the winner of $\Lambda_i$ is in $C$, and
  elements of both $B$ and $A$ are represented amongst the winners of the
  other $\Lambda_j$.  Therefore, the previous lemma's argument implies that
  the winner of $\Lambda_i^2$ is from the set which defeats $C$, i.e., $B$.
\end{proof}

We are now ready to describe our main construction for this section. Start 
with the voting tree $\Lambda^2_j$ and replace each leaf labeled $j$ with
a copy of the voting tree $\Phi$ from Definition \ref{def:shuffle-tree}
and call the resulting tree $\Psi$.

\medskip

\begin{proof}[Proof of Theorem \ref{thm:manipulator}]
  Let $T$ be the perfect manipulator tournament given as input.  With
  respect to $T$, when $\Phi$ is evaluated we will be given a vertex
  which is not in $C$. Now when something which is not $C$ is feed as 
  input into $\Lambda^2_i$ then we get something which is not in $A$
  by Corollary \ref{cor:one-against-all-2}. So, the output of the whole tree 
  $\Psi$ will not be $\alpha$ as desired. 
\end{proof}

\section{Arithmetic circuits}

In this section, we construct more complex chains of voting trees for $n=3$
which can perform arithmetic modulo three.  An essential caveat is that
such computation is clearly impossible when the input tournament is
transitively oriented, and so we design our circuits to succeed no matter
which non-transitive tournament is input.  The two such tournaments are
shown in Figure~\ref{fig:tournaments}.

\begin{figure}[h]
\includegraphics[scale=.25]{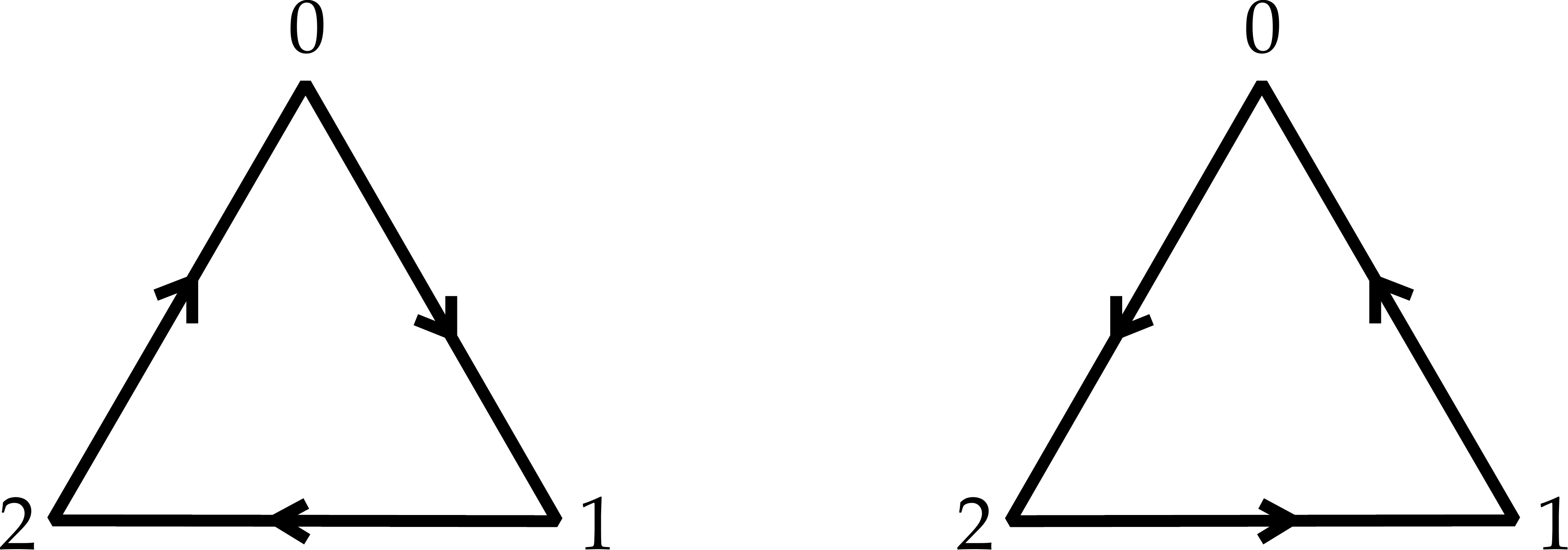} 
\centering
\label{fig:tournaments}
\caption{\footnotesize There are two non-transitive tournaments on three
vertices. We call the tournament on the left the \emph{clockwise}\/
tournament, and the tournament on the right the \emph{counter-clockwise}\/
tournament.}
\end{figure}

We will now create trees of increasing complexity by developing robust
``gates'' and putting them together.  However, this time, because our aim
is to send outputs of gates as inputs into other gates, we will design
voting trees whose leaves are labeled with both tournament vertices (0, 1,
2) and variables $X$, $Y$, etc.  Our basic building block, which we call a
\emph{yield gate}, is again along the lines of Definition
\ref{def:against-S}, and is illustrated in Figure~\ref{fig:beatsGate}.

\begin{figure}[h]
\includegraphics[scale=.25]{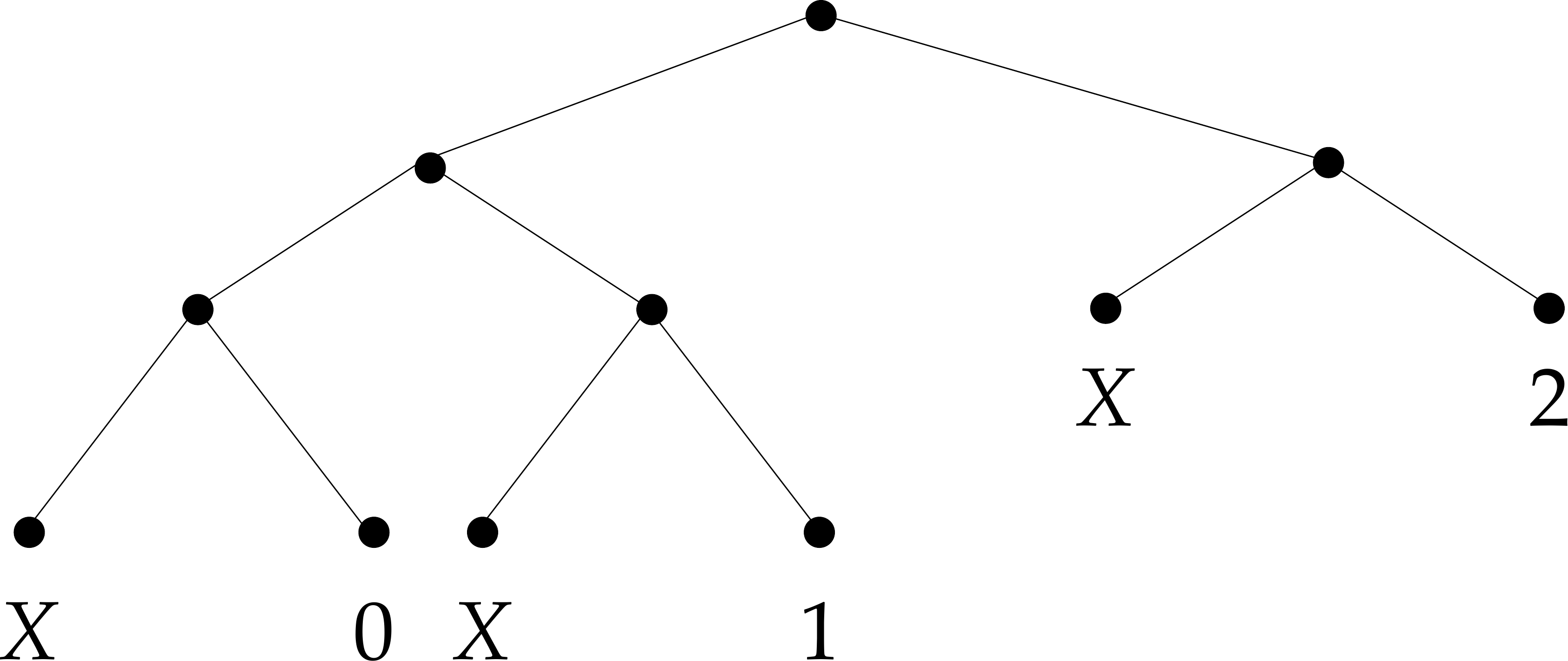} 
\centering
\caption{\footnotesize A \emph{yield gate}, a tree which outputs a candidate which beats the
input $X$.}
\label{fig:beatsGate}
\end{figure}

The same argument as used in the proof of Lemma \ref{lem:one-against-all}
produces the following conclusion.

\begin{corollary}
  For any non-transitive tournament $T$ and any vertex $x \in T$, if all
  labels $X$ are replaced with $x$ in the yield gate, and the tree is
  evaluated with respect to $T$, then the winner is the unique vertex which
  beats $x$.
  \label{cor:yield}
\end{corollary}

\subsection{Addition}

Next, we combine yield gates to create a negative sum gate, which takes $X$
and $Y$ as inputs, and outputs $-X-Y$ (modulo 3) with respect to any
non-transitive tournament.  A simple match between two yield gates, shown
in Figure~\ref{fig:almostSub}, almost achieves this behavior.

\begin{figure}[h]
\includegraphics[scale=.25]{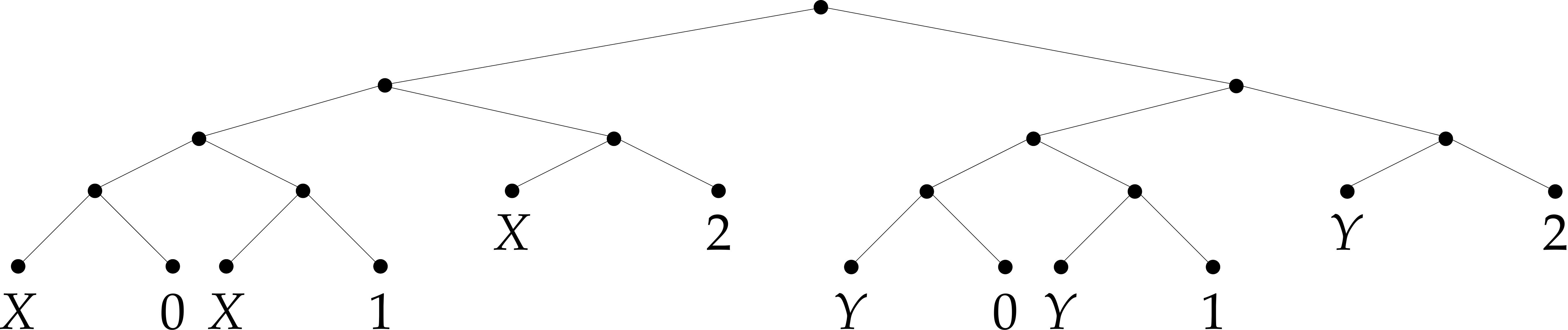} 
\centering
\caption{\footnotesize A simple combination of two yield gates.}
\label{fig:almostSub}
\end{figure}

Let us analyze the behavior of this tree with respect to the various
tournaments and inputs $X$ and $Y$.  If $X=Y$, then both halves of the tree
return the same output, which is $X-1$ for the clockwise tournament, and
$X+1$ for the counter-clockwise tournament.  Now consider the case where
$X\neq Y$, and let $Z$ be the third vertex (so $Z = \{0, 1, 2\} \setminus
\{X, Y\}$). In the clockwise tournament, the two halves of the tree would
return $X-1$ and $Y-1$. If $X=Y-1$ then $X-1$ would beat $Y-1$, producing
$X-1$ as the winner.  Here, $X-1 = Z$ because $X-1$ is neither $X$ nor $Y$.
On the other hand, if $Y=X-1$ then we would have the result be $Y-1$ which
is again $Z$. Doing the same calculations for the counter-clockwise
tournament, we get that $X+1$ wins if $X=Y+1$, and $Y+1$ wins if $Y=X+1$.
In summary, whenever $X \neq Y$, regardless of the tournament's
orientation, the output is $Z$, which equals $-X-Y$ because $\{X, Y, Z\} =
\{0, 1, 2\}$ and $0 + 1 + 2 = 0$.  So the gate shown in
Figure~\ref{fig:almostSub} gives us the desired outcome of $-X-Y$ when
$X\neq Y$. We will now stack several gates of this type together to create
our desired gate, as illustrated in Figure \ref{fig:negSum}. 

\begin{figure}[h]
\includegraphics[scale=.25]{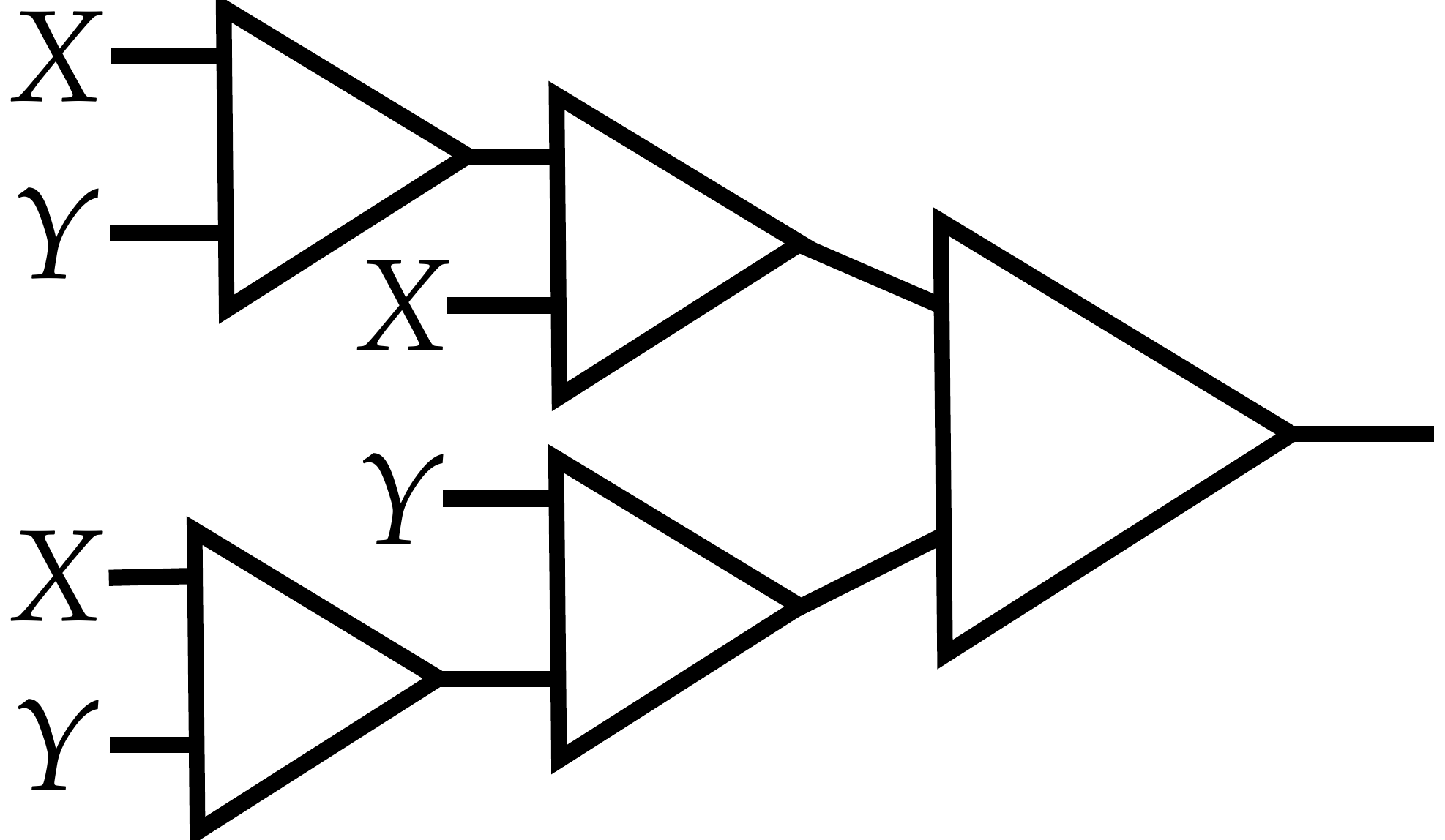} 
\centering
\caption{\footnotesize This circuit computes $-X-Y$.  Each triangular gate
  in this arrangement is a copy of the gate shown in Figure
  \ref{fig:almostSub}.  The two inputs on the left of each gate correspond
  to $X$ and $Y$, and the output on the right is the output of the tree.
  When the output of a structure is plugged into (say) the $X$-input of
  another gate, all copies of $X$ in the latter gate are replaced with
  identical copies of the structure.}
\label{fig:negSum}
\end{figure}

\begin{lemma}
  The gate in Figure \ref{fig:negSum} computes $-X-Y$.
  \label{lem:negSum}
\end{lemma}

\begin{proof}
  First we consider the case $X\neq Y$, and let $Z = \{0, 1, 2\} \setminus
  \{X, Y\}$. So, the output of both leftmost trees is $-X-Y = Z$. Now
  $Z\neq X$ and $Z\neq Y$ so the next layer of trees output $Y$ and $X$
  respectively.  The last gate receives $Y$ and $X$ as input, so outputs
  $-X-Y$ as desired. 

  Now we consider the case $X=Y$, starting with the clockwise tournament.
  By the discussion above, both leftmost trees output $X-1$.  Now the
  second layer of trees receive $X$ and $X-1$ as input, and therefore both
  give $X+1$ as output. The last tree now receives $X+1$ for both inputs,
  and so gives $(X+1)-1$ as output.  This is our desired outcome as $-X-X =
  X$ (modulo 3). In the final case, where $X=Y$ and the tournament is
  counter-clockwise, similar calculations yield that the last layer outputs
  $(X-1)+1=X = -X-X$ as desired. 
\end{proof}

Observe that if we send the constant 0 as the input for $Y$ in the $-X-Y$
gate, we produce a unary negation gate.  If we then feed the output of the
negative sum tree through this negation tree, then the output becomes the
ordinary sum $X+Y$, which was the objective of this section. 

\begin{corollary}
  There is a gate which takes inputs $X$ and $Y$, and outputs their sum
  $X+Y$ (in $\mathbb{F}_3$), when evaluated with respect to any
  non-transitive tournament on three vertices.
\end{corollary}

The gates constructed thus far will be used repeatedly in later sections,
and Figure~\ref{fig:gates} shows how they will be compactly represented in
the remainder of this paper.

\begin{figure}[h]
\includegraphics[scale=.25]{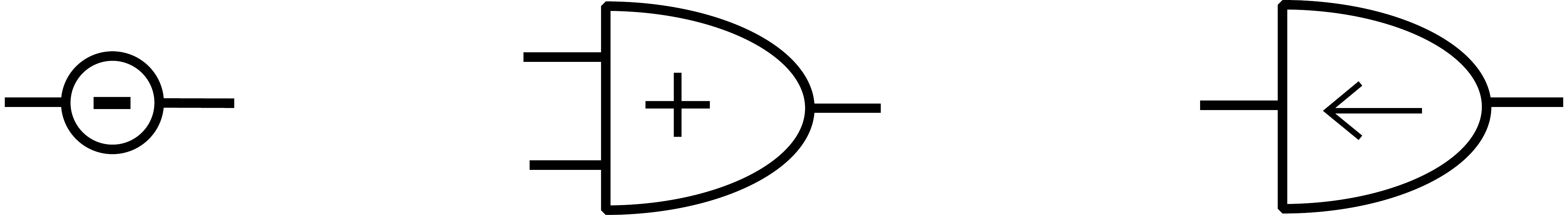} 
\centering
\caption{\footnotesize These symbols will represent (from left to right)
  the negation gate, addition gate, and yield gate. Inputs enter the left
  side of each gate and outputs exit via the right side of the gate.}
\label{fig:gates}
\end{figure}

\subsection{Squaring}

Having created our addition gate, we move toward multiplication.  To this
end, in this section we develop a unary gate which computes the function $X
\mapsto X^2$, regardless of which non-transitive 3-vertex tournament is
used.  Observe that the squaring function sends 0 to 0, and sends both 1
and 2 to the same value 1.  We start by creating a gate (shown in
Figure~\ref{fig:firstHalf}) with similar properties.

\begin{figure}[h]
\includegraphics[scale=.25]{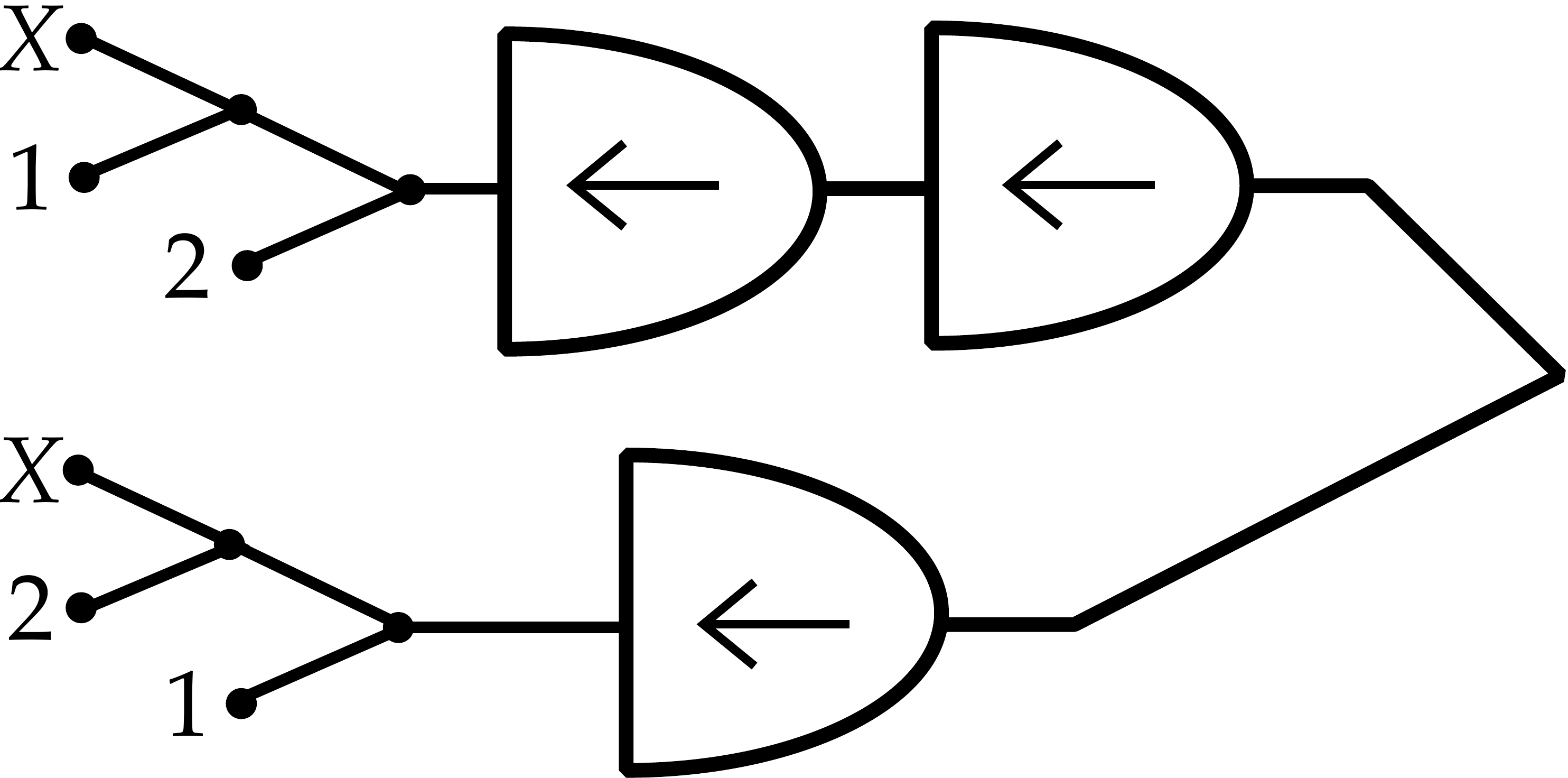} 
\centering
\caption{\footnotesize The first half of the squaring gate.  The final
merging on the right side represents an ordinary voting tree comparison
between the outputs of the top and bottom chain of gates.}
\label{fig:firstHalf}
\end{figure}

\begin{lemma}
  The gate in Figure~\ref{fig:firstHalf} sends $0 \mapsto 0$, $1 \mapsto
  2$, and $2 \mapsto 2$ on the clockwise tournament, and $0 \mapsto 2$, $1
  \mapsto 1$, and $2 \mapsto 1$ on the counter-clockwise tournament.
  \label{lem:squaring-1}
\end{lemma}

\begin{proof}
  Consider the clockwise tournament.  On the top left of the gate, if $X=0$
  then 0 beats 1 and then 2 beats 0.  Otherwise, the top left is simply the
  winner between 1 and 2, which in this case is 1. Feeding this through the
  yield gate twice we get the top produces 0 when 0 is input, and 2 when
  either 1 or 2 are input. On the bottom left of the tree, if $X=0$ then 2
  beats 0 and then 1 beats 2.  Otherwise, the bottom left is again the
  winner between 1 and 2 which is again 1.  So, for the clockwise
  tournament, the bottom always feeds 1 through a yield gate and produces
  0, regardless of the input $X$.  Matching the outputs of the top and the
  bottom, and evaluating with respect to the clockwise tournament, we find
  that the gate in Figure~\ref{fig:gates} outputs 0 if $X=0$ and outputs 2
  otherwise, as claimed.

  It remains to consider the counter-clockwise tournament.  On the bottom
  left of the tree, if $X=0$ then 0 beats 2 and then 1 beats 0. Otherwise,
  the bottom left is simply the winner between 1 and 2, which is 2. Feeding
  this through the yield gate, we find that the bottom is 2 when 0 is
  input, and 0 when either 1 or 2 are input. On the top left of the tree,
  if $X=0$ then 1 beats 0 and then 2 beats 1. Otherwise, the top left is
  again the winner between 1 and 2 which is again 2.  So, for the
  counter-clockwise tournament, the top always feeds 2 through a yield gate
  twice and produces 1, regardless of the input $X$.  Combining the top and
  the bottom we conclude that the gate outputs 2 if $X=0$ and outputs 1
  otherwise.
\end{proof}

At this point we are relatively close to getting the desired outcome, as
for each non-transitive tournament, the inputs 1 and 2 produce one result,
and the input 0 produces a different one.  We complete our squaring gate by
sending this output into the gate shown in Figure \ref{fig:secondHalf}.

\begin{figure}[h]
\includegraphics[scale=.25]{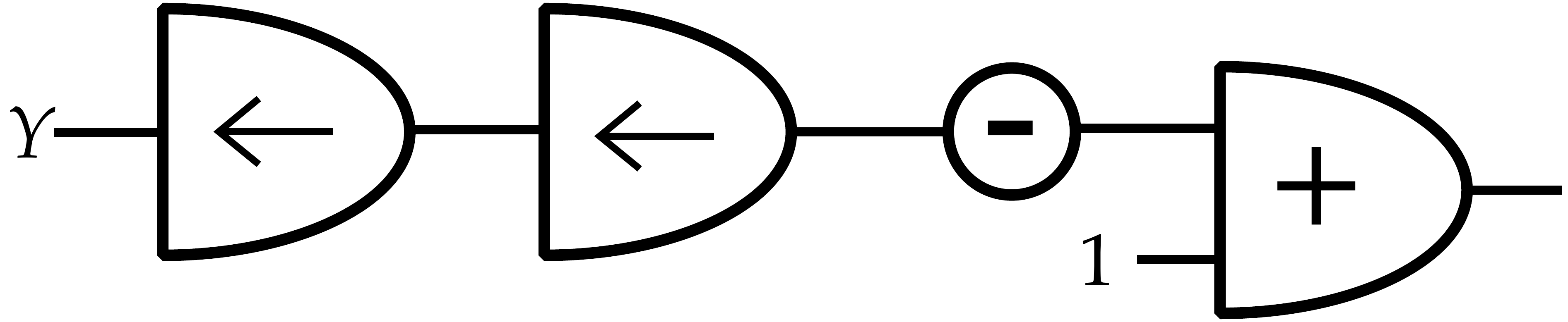} 
\centering
\caption{\footnotesize The second half of the squaring gate.  Here, $Y$ is
used to represent the input to this half, so there is no confusion with the
input $X$ to the first half (from Figure \ref{fig:firstHalf}).}
\label{fig:secondHalf}
\end{figure}

\begin{lemma}
  The composition of the gates in Figures \ref{fig:firstHalf} and
  \ref{fig:secondHalf} computes the function $X \mapsto X^2$.
  \label{lem:square-gate}
\end{lemma}

\begin{proof}
  Let $Y$ be the output of the gate in Figure \ref{fig:firstHalf} when $X$
  is input.  In the case of the clockwise tournament, $Y=0$ if $X=0$ and
  $Y=2$ otherwise. Feeding this through the two yield gates, we get 1 when
  $X=0$ and 0 otherwise. In the case of the counter-clockwise tournament,
  $Y=2$ if $X=0$ and $Y=1$ otherwise. When this passes through two yield
  gates, then we again get 1 when $X=0$ and 0 otherwise. So, regardless of
  the tournament used, the result after the two yield gates is 1 when
  $X=0$, and 0 otherwise.  We wish to have 0 when $X=0$ and 1 otherwise.
  This is achieved by subtracting our current result from 1, which is
  realized by the negation and addition gates in the right half of Figure
  \ref{fig:secondHalf}.
\end{proof}

\subsection{Multiplication}

Now that we have an addition gate and a squaring gate, we build a
multiplication gate via the identity
\begin{displaymath}
X^2+Y^2-(X+Y)^2=-2XY \,,
\end{displaymath}
which equals $XY$ modulo 3.  The gate in Figure~\ref{fig:multiply} computes
this formula, using the gates that we developed earlier in this section. 

\begin{figure}[h]
\includegraphics[scale=.25]{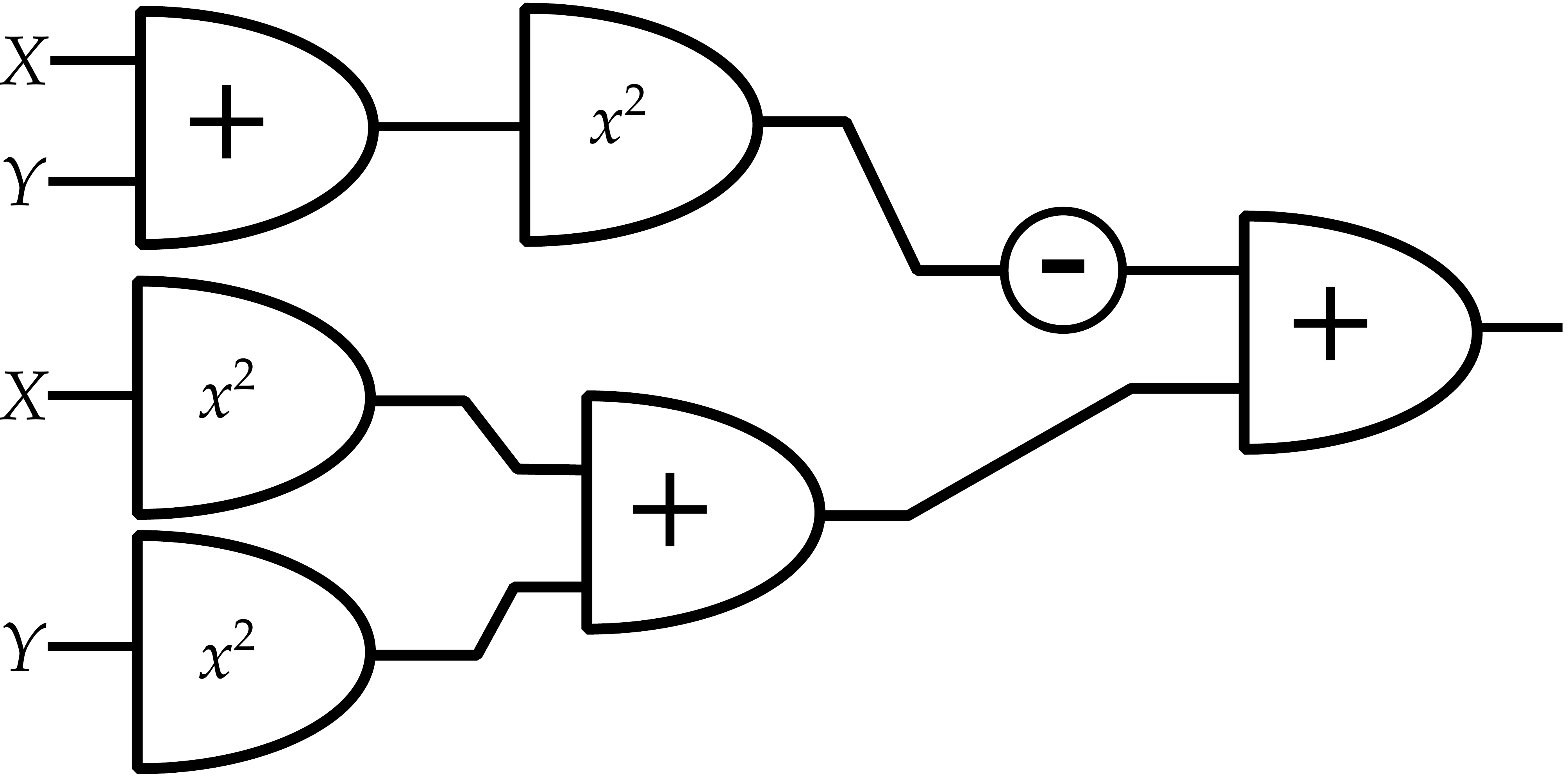} 
\centering
\caption{\footnotesize The multiplication gate.}
\label{fig:multiply}
\end{figure}

\section{Concluding Remarks}

%

This paper presented several constructions which use voting trees to
compute natural and non-trivial functions.  In particular, we substantially
improve upon the previous best performance guarantee of $\log_2 n$,
constructing trees which achieve order $\sqrt{n}$.  We presented our
arithmetic gates in our final section because they were initially developed
for lower bound constructions for the performance guarantee, and produced
our first construction which improved the lower bound.  We also feel that
they may be of independent interest, as they demonstrate that the
computational power of voting trees is sufficiently rich to perform
standard ternary arithmetic.

\end{document}